\documentclass{amsart}

\usepackage[T1]{fontenc}

\usepackage{amsmath}
\usepackage{amsfonts}
\usepackage{amsthm}
\usepackage{amssymb}

\usepackage{microtype}

\usepackage{diagbox}
\usepackage{booktabs}

\usepackage{svg}
\usepackage{float}
\usepackage{pinlabel}
\usepackage{caption}
\usepackage{subcaption}
\usepackage{mathtools}

\DeclarePairedDelimiter\abs{\lvert}{\rvert}

\theoremstyle{cupthm}
\newtheorem{thm}{Theorem}[section]

\newtheorem{lemma}[thm]{Lemma}

\theoremstyle{cuprem}

\numberwithin{equation}{section}
\newtheorem{conj}[thm]{Conjecture}
\newtheorem{quest}[thm]{Question}

\newtheorem*{question*}{Question}

\newcommand{\R}{\mathbb{R}}

\newcommand{\TY}{\nabla{} Y}
\newcommand{\YT}{Y \nabla}

\newcommand{\define}[1]{\emph{#1}}

\newcommand{\F}{\mathcal{F}}
\newcommand{\Fd}{\mathcal{F}_{\Delta}}

\begin{document}


\title{Family sizes for complete multipartite graphs}

\author{Danielle Gregg}
\email{greggd@union.edu}
\address{Department of Mathematics,
  Union College, Bailey Hall 202,
  Schenectady, New York 12308}

\author{Thomas W.\ Mattman}
\email{TMattman@CSUChico.edu}
\address{Department of Mathematics and Statistics,
California State University, Chico,
Chico, CA 95929-0525}

\author{Zachary Porat}
\email{zporat@wesleyan.edu}
\address{ Department of Mathematics and Computer Science,
  Wesleyan University,
  Middletown, CT 06459}

\author{George Todd}
\email{gtodd1@udayton.edu}
\address{Department of Mathematics,
  University of Dayton,
  Dayton, OH, 45469}


\begin{abstract}
  The obstruction set for graphs with knotless embeddings is not known, but a recent paper of Goldberg, Mattman, and Naimi indicates that it is quite large.
  Almost all known obstructions fall into four Triangle-Y families and they ask if there is an efficient way of finding or estimating the size of such graph families.
  Inspired by this question, we investigate the family size for complete multipartite graphs.
  Aside from three families that appear to grow exponentially, these families stabilize: after a certain point, increasing the number of vertices in a fixed part does not change family size.
\end{abstract}


\maketitle

\section{Introduction}

This paper is inspired by the question of Goldberg et al.~\cite{GMN}:

\begin{question*}[\cite{GMN}, Question 4]
  Given an arbitrary graph, is there an efficient way of finding, or at least estimating, how many cousins it has?
\end{question*}

We show that, in the case of a complete multipartite graph, there is quite a lot one can say about its family size.

For us, graphs are finite, undirected, and simple.
We say that $H$ is a \define{minor} of $G$ if $H$ is obtained by contracting edges in a subgraph of $G$.
The Graph Minor Theorem of Robertson and Seymour~\cite{Seymour}, perhaps the most important result in graph theory, says that any property of graphs that is inherited by minors has a finite obstruction set.
Here, we are primarily interested in topological properties of graphs.

For example, the obstruction set for graph planarity (embedding a graph in the plane, or the sphere, with no crossings) contains only the complete multipartite graphs $K_5$ and $K_{3,3}$.
In general, while there will be a finite set of obstructions for embeddings into any surface, the obstruction set may be quite large.
For example, it is known that for the torus there are at least sixteen thousand obstructions (see~\cite{GMC}) and it is expected that the number grows quickly with genus after that.

Robertson, Seymour, and Thomas~\cite{RST} proved that the obstruction set for graphs with a linkless embedding (an embedding in $\mathbb{R}^3$ that contains no nontrivial link) is the seven graphs obtained from $K_6$ by \define{Triangle-Y} and \define{Y-Triangle} moves, see Figure~\ref{figTY}.
In addition to $K_6$, this set of seven graphs, called the Petersen Family, also contains $K_{1,3,3}$ and the Petersen graph.
It was shown in \cite{CG} that $K_7$ is an obstruction for knotless embedding (embeddings in $\mathbb{R}^3$ with no non-trivial knot).
Although the obstruction set for graphs with knotless embeddings is not known, recent work of Goldberg, Mattman, and Naimi indicates it is quite large \cite{GMN} and that completing the set may be beyond current theory.

\begin{figure}[htb]
  \labellist{}
  \small\hair 2pt
  \pinlabel $a$ at 46 -6
  \pinlabel $b$ at 2 90
  \pinlabel $c$ at 90 90
  \pinlabel $a$ at 272 -6
  \pinlabel $b$ at 228 90
  \pinlabel $c$ at 316 90
  \pinlabel $v$ at 265 43
  \pinlabel triangle-Y at 160 68
  \pinlabel Y-triangle at 160 33
  \endlabellist{}
  \centering
  \includegraphics[scale=0.6]{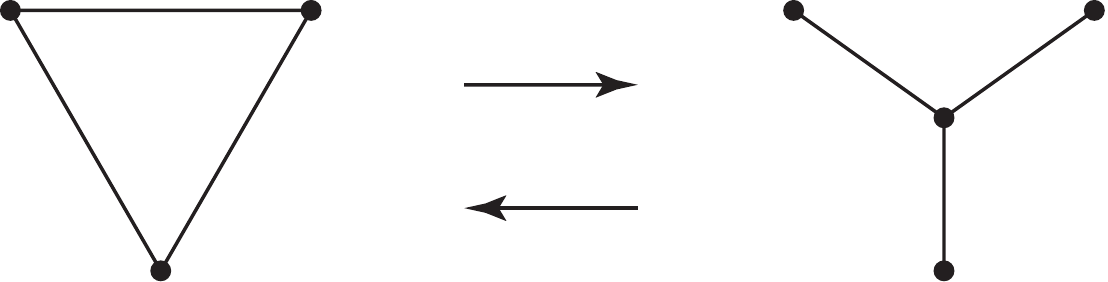}
  \caption{The $\TY$ and $\YT$ moves.}\label{figTY}
\end{figure}

While the Graph Minor Theorem guarantees finite obstruction sets, we often have no way of bounding, or even estimating that finite number.
In the case of linkless embedding, the obstructions belong to a single family related by Triangle-Y and Y-Triangle moves.
Similarly, for knotless embedding, all but three of the known obstructions fall into one of four families~\cite{FMMNN}.
Bounding the complete set of knotted obstructions is beyond us for the moment, but a method for estimating the size of graph families is a positive step in that direction. 

As in Figure~\ref{figTY}, a \define{Triangle-Y} or $\TY$ move deletes the edges of a 3-cycle $abc$ in graph $G$ and adds a new vertex $v$ and the three edges $av$, $bv$, $cv$ to create a new graph $H$.
We call the inverse operation, from $H$ to $G$, a \define{Y-Triangle} or $\YT$ move.
Let $\#E(G)$ be the number of edges in the graph $G$, called the graph's \emph{size}, and note that these moves do not change graph size: $\#E(G) = \#E(H)$.
If a graph $H$ is obtained from $G$ by a sequence of zero or more $\YT$ and $\TY$ moves, we say $H$ and $G$ are \define{cousins}.
The set of cousins of a graph $G$ is known as $G$'s \define{family}, denoted $\F(G)$.
Every graph in $\F(G)$ has the same size as $G$. In the current paper we seek to estimate $|\F(G)|$, the number of graphs in $\F(G)$, which we will call $G$'s \define{family size}.

Since the $\TY$ and $\YT$ moves preserve important topological properties of a graph, these families are significant in the study of spatial graphs, or embeddings of graphs in $\R^3$.
For example, $\YT$ preserves planarity, and more generally, preserves $n$-apex provided the vertex $v$ is not part of an apex set, see~\cite{MP}.
As in that paper, we say that a graph is \define{$n$-apex} if it can be made planar by deletion of $n$ or fewer vertices.
Sachs~\cite{S} observed that $\YT$ preserves linkless embeddings, and, essentially the same argument, shows that it also preserves knotless embeddings.

As mentioned above, linkless embedding is characterized by the family of the Petersen graph.
Sachs~\cite{S} saw that the graphs in the Petersen family are obstructions for linkless embedding and conjectured that those seven constituted a complete list of obstructions; this was confirmed in~\cite{RST}.
In addition to the Petersen graph, the complete graph $K_6$ and the complete tripartite graph $K_{1,3,3}$ are in this family, so we can denote it as $\F(K_6)$ or $\F(K_{1,3,3})$.
Since $\F(K_6)$ is closed under $\TY$ moves, that move also preserves linkless embeddings, see~\cite{FN}.

On the other hand, Flapan and Naimi~\cite{FN} pointed out that, in general, $\TY$ does not preserve knotless embeddings. 
Nonetheless, almost all of the 264 known obstructions belong to the four families $\F(K_7)$,
$\F(K_{1,1,3,3})$, $\F(E_9+e)$, and $\F(G_{9,28})$, see \cite{FMMNN}.

In \cite{GMN}, the authors note that family size shows considerable variation.
For example, they contrast $G_{14,25}$, a graph of order 14 and size 25, whose family size is at least several hundreds of thousands, with an obstruction discovered by Foisy, of order 13 and size 30, whose family size is one.
In the current paper, we investigate what can be said if we restrict attention to the families of complete multipartite graphs.
We have already seen how the families of $K_6$, $K_7$, and $K_{1,1,3,3}$ are important in characterizing linkless and knotless embeddings.
These ideas are generalized in \cite{MP} where the authors present evidence that the graphs in $\F(K_n)$ and $\F(K_{1^n,3^2})$ are obstructions for the $n$-apex property.
Here, $K_{1^n,3^2}$ denotes the complete multipartite graph with two parts of three vertices each and a further $n$ parts, each of a single vertex.

In summary, the families of complete multipartite graphs have already shown their utility in the study of spatial graphs.
Moreover, since any graph can be made complete multipartite through the addition of edges, information about the family size of complete multipartite graphs can be parlayed into estimates for other graphs. 
For example, we've mentioned $E_{9+e}$ and $G_{9,28}$ as important obstructions for knotless embedding.
The family size of $E_{9+e}$ is 110, which is similar to the size 71 for the graph $K_{3^3}$ that has five more edges.
For $G_{9,28}$, whose family size is 1609, we can compare with $K_{1,2^4}$, which has four extra edges and family size 1887.

\section{Results}

Our main observation is that the sizes of families of complete multipartite graphs stabilize as the number of vertices in any fixed part increases.

\begin{thm}\label{thmmain}
  Let $1 \leq a_1 \leq \cdots \leq a_n$ and $e = \#E(K_{a_1, \ldots, a_{n-1}})$. If $a_1 + \cdots + a_{n-1} > 6$ and $a_n \geq e$, then
  \[
      \abs*{\mathcal{F} (K_{a_1, \ldots, a_n})} = \abs*{\mathcal{F} (K_{a_1, \ldots, a_{n-1}, e})}.
  \]
\end{thm}

For tripartite graphs, we verify stabilization even when the sum of the parts does not exceed six, with one exception.

\begin{thm}\label{thmtri}
  Let $1 \leq a \leq b \leq c$, $(a,b)  \neq (1,2)$, and $c \geq d = \max(4,ab)$.
  Then $\abs*{\F(K_{a,b,c})} = \abs*{\F(K_{a,b,d})}$.
\end{thm}

For bipartite graphs, the family size is generally one and it is also relatively small for $K_{1,b,c}$.

\begin{thm}\label{thmbip}
  For $K_{x,y}$, if $x \neq 3$ and $y \neq 3$, then $\abs*{\mathcal{F}(K_{x,y})} = 1$.
\end{thm}

\begin{thm}\label{thmK1bc}
  Let $6 \leq b \leq c$. Then $\abs*{\mathcal{F} (K_{1,b,c})} = 1 + b$.
\end{thm}

%

For $K_{2,b,c}$ we also have a lower bound in terms of partitions, which closely follows the observed growth of $\abs*{\F(K_{2,b,c})}$.
Let $P(x,y,z)$ denote the set of partitions of $z$ into two parts, the first bounded by $x$ and the second by $y$:

\[
  P(x,y,z) = \{ \, (m,n) : 0 \leq m \leq x, 0 \leq n \leq y, m + n = z \, \}.
\]

Define $g(b,c)$ by

\[
  g(b,c) = 5 + \sum_{i=2}^b \sum_{j=0}^i \left( \abs*{P(i, b-i, j)} \cdot \abs*{P(i, c-i, j)} \right)\,.
\]
 
\begin{thm}\label{thmK2bc}
  If $c > b \geq 3$, then $g(b,c) \leq \abs*{\F(K_{2,b,c})}$.
\end{thm}

Although, the family sizes of complete multipartite graphs tend to stabilize, we've encountered three types of graphs that do not follow this pattern.
For these we propose instead estimates of the family sizes supported by computational observations.

\begin{quest}\label{q1}
  Does $|\F(K_n)|$ grow as
	\[
		f(n) = \frac{6}{5} {\left( 2 \pi \right)}^{3/2} e^{{(n-7)}^2/2}\,?
	\]
\end{quest}

Table~\ref{fig:Kgrowth} gives the estimated and actual values of $\abs*{\F(K_n)}$ for $8 \leq n \leq 12$.

\begin{table}[ht]
	\centering
	\begin{tabular}{ccc}
		\toprule
		$n$ & $f(n)$ & Actual \\
		\midrule
		8 & 31.2 & 32 \\
		9 & 139.7 & 163 \\
		10 & 1701.3 & 1681 \\
		11 & 56,338.7 & 56,461 \\
		12 & 5,071,450 & 5,002,315 \\
		\bottomrule
	\end{tabular}
	\caption{Estimates for the Family Size of $K_n$ Versus Actual\label{fig:Kgrowth}}
\end{table}

\begin{quest}\label{ques2}
  Is $\frac{8}{3} e^{\frac35y} > \abs*{\F(K_{3,y+3})}$ for $y \geq 4$?
\end{quest}

\begin{quest}\label{ques3}
  Is $\frac{16}{3} e^{\frac23c} < \abs*{\F(K_{1,2,c+3})}$ for $c \geq 1$?
\end{quest}

In the next section, we introduce some additional terminology and prove Theorem~\ref{thmtri}.
In Section 3, we prove our main theorem, Theorem~\ref{thmmain}.
Section 4 is devoted to the three families that do not appear to stabilize, including motivation for the estimates given as part of our three questions.
We prove Theorems~\ref{thmbip}, \ref{thmK1bc}, and \ref{thmK2bc} in Section 5, where we also state a conjecture for multipartite graphs.

%
%
\section{Families of tripartite graphs stabilize.}

Let $K_{a,b,c}$ denote the complete $a,b,c$ tripartite graph, where $1 \leq a \leq b \leq c$. Let $\Fd (G)$ denote the \emph{family of descendants} of $G$, the graphs that can be obtained from graph $G$ by a sequence of $\TY$ moves, along with $G$ itself.
We call an element of $\Fd(G)$ a \emph{descendant} of $G$.
We argue that, with the exception of $(a,b) = (1,2)$, the sizes of these two families stabilize for $c \geq ab$. 
We conclude this section with a proof of Theorem~\ref{thmtri}.

Let $G = K_{a,b,c}$ and $\{A, B, C\}$ be the partition of $V(G)$ with $\abs*{A} = a$, $\abs*{B} = b$, and $\abs*{C} = c$.
The triangles of $K_{a,b,c}$ are $(v,w,x)$ with $v \in A$, $w \in B$, and $x \in C$ and every such triple of vertices gives a triangle.
Let $H$ be the child of $G$ born of a $\TY$ move at $(v,w,x)$.
Then $V(H) = V(G) \cup \{y\}$, where $y$ is a degree three vertex with neighborhood $N(y) = \{v,w,x\}$.
We will refer to $y$ as a \define{trivial degree three vertex} since a $\YT$ move at $y$ simply recovers the graph $G$ and reverses the $\TY$ move that brought us to $H$ in the first place.
Since none of the edges of $(v,w,x)$ remain in $H$, $y$ is not part of a triangle in $H$.

More generally, any descendant $H$ of $G$ is born of a sequence of $\TY$ moves at edge-disjoint triangles $(v_1,w_1,x_1), \ldots, (v_n, w_n, x_n)$.
These result in a sequence of trivial vertices $y_1, \ldots, y_n$ none of which are vertices of a triangle in $H$.
Conversely, $\TY$ moves at any set of edge-disjoint triangles in $G$ produces one of its descendants.
\begin{lemma} \label{lemFdeqF}
  Let $1 \leq a \leq b \leq c$. If $a+b > 6$, then
  \[
    \abs*{\Fd (K_{a,b,c})} = \abs*{\F (K_{a,b,c})} \,.
  \]
\end{lemma}
\begin{proof}
  The idea is that $\TY$ moves will produce only trivial degree three vertices; the only $\YT$ moves in this family simply reverse earlier $\TY$ moves.

  The vertices of least degree are those in the $C$ part, of degree $a+b$. Let $x \in C$.
  A $\TY$ move on a triangle at $x$ replaces two of its edges with one. This means that $\TY$ moves can at most halve the degree of $x$.
  Since $a+b > 6$, the degree of $x$ will never drop to three.
  As the vertices in the $A$ and $B$ parts have even higher degree, the only degree three vertices in a descendant of $K_{a,b,c}$ are the trivial ones.
\end{proof} 

\begin{lemma}\label{lemb3}
  Let $1 \leq a \leq b \leq c$. If $b > 3$, then
  \[
    \abs*{\Fd (K_{a,b,c})} = \abs*{\F (K_{a,b,c})} \,
  \]
\end{lemma}
\begin{proof}
  The previous lemma treats the case where $a+b > 6$, so we may assume $a < 3$.
  Again, we'll argue that the only degree three vertices are trivial.

  Suppose $a=1$ and let $v$ denote the unique vertex in that part of the graph.
  If $x$ is in $B$ or $C$, then, in a descendant of $K_{a,b,c}$ there is at most one $\TY$ move involving $x$ and so the degree of $x$ decreases by one at most.
  Since $3<b\leq c$, the degree of $x$ remains greater than three in the descendant.
  As for $v$, it starts with a degree exceeding six and is at most halved by $\TY$ moves.
  So, the only degree three vertices in a descendant are trivial.

  If $a = 2$, the argument is similar.
  As in the previous case, even after halving, vertices $v$ in the $A$ part have degree greater than three.
  As for a vertex $x$ in $B$ or $C$, it can be involved in at most two triangles.
  But the degree of $x$ is at least six, so removing two still leaves it above three.
\end{proof}

\begin{thm}\label{thmFstab}
  Let $1 \leq a \leq b$. If $c \geq ab$, then
  \[
    \abs*{\Fd (K_{a,b,c})} = \abs*{\Fd (K_{a,b,ab})} \,.
  \]
\end{thm}

\begin{proof}
  Let $G = K_{a,b,c}$ with $c \geq ab$.
  As discussed above, any descendant $H$ of $G$ is the result of a sequence of $\TY$ moves on edge-disjoint triangles, $(v_1,w_1,x_1), \ldots, (v_n,w_n,x_n)$, and the introduced degree three vertices $y_1, \ldots, y_n$ are not part of a triangle in $H$.
  In other words, there is a correspondence between elements of $\Fd (G)$ and sequences $(v_1,w_1,x_1), \ldots, (v_n,w_n,x_n)$ of triangles in $G$.

  As the triangles in such a sequence must be edge-disjoint, the maximum length $n$ of such a sequence is $ab$, the number of edges in the induced complete bipartite graph $K_{a,b}$.

  This leads to a bijection between the elements of $\Fd (K_{a,b,ab})$ and $\Fd (G)$.
  If $H$ is a descendant of $G$, let $(v_1,w_1,x_1), \ldots, (v_n,w_n,x_n)$ be the associated sequence of edge-disjoint triangles.
  Extend the labeling of vertices of $C$ so that $C = \{x_1, \ldots, x_n, x_{n+1}, x_{n+2}, \ldots, x_c\}$.
  By deleting vertices $\{x_{ab+1}, x_{ab+2}, \ldots, x_c\}$ we identify $H$ with an element $H'$ of $\Fd (K_{a,b,ab})$.
  Conversely, by adding vertices $x_{ab+1}, x_{ab+2}, \ldots, x_{c}$, adjacent to each vertex in $A$ and $B$, any graph $H' \in \Fd (K_{a,b,ab})$ becomes
  a $H \in \Fd (G)$.
\end{proof}

Lemmas~\ref{lemFdeqF} and \ref{lemb3} leave open five cases, besides $(1,2)$. The following three lemmas handle these remaining cases.

\begin{lemma}\label{lem3cases}
  Let $1 \leq a \leq b \leq c$ and $c \geq d = \max(4,ab)$.
  Then $\abs*{\F(K_{a,b,c})} = \abs*{\F(K_{a,b,d})}$ in case $(a,b) \in \{ (1,1), (1,3), (2,2) \}$.
\end{lemma}

\begin{proof}
  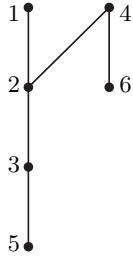
\begin{figure}[htb]
    \labellist{}
    \small\hair 2pt
    \pinlabel $1$ at -6 150
    \pinlabel $2$ at -6 105
    \pinlabel $3$ at -6 55
    \pinlabel $5$ at -6 5
    \pinlabel $4$ at 64 150
    \pinlabel $6$ at 64 105
    \endlabellist{}
    \includegraphics[scale=0.6]{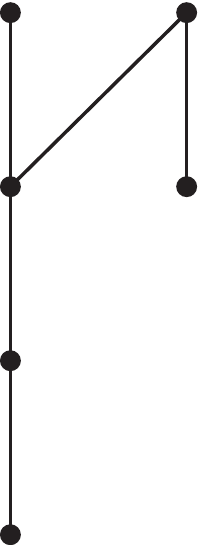}
    \centering
    \caption{The $K_{1,3,4}$ family.\label{figK134}
  }
  \end{figure}

  If $(a,b) = (1,1)$, then $d = 4$.
  Up to symmetry, there is only one triangle in $K_{1,1,c}$ and applying the $\TY$ move leaves a graph that has only one degree three vertex, which is trivial.
  Thus $\abs*{\F(K_{1,1,c})} = \abs*{\F(K_{1,1,4})} = 2$.

  If $(a,b) = (1,3)$, then $d = 4$.
  There are six graphs in $\F(K_{1,3,4})$, illustrated schematically in Figure~\ref{figK134}.
  Graphs at the same height have the same number of vertices (they all have the same number of edges).
  We will argue that, if $c \geq 4$, $\F(K_{1,3,c})$ has the same structure and the same size, six.

  Graph 1 in the figure is $K_{1,3,4}$, and the three graphs below it, 2, 3, and 5, round out $\Fd(K_{1,3,4})$.
  More precisely, in addition to $K_{1,3,4}$ itself, there are three descendants corresponding to the three edges in $K_{1,3}$, the subgraph induced by the vertices in part A and B.
  Each of those three edges can be completed to a triangle using a vertex of part C, and, there are no other (edge-disjoint) triangles in $K_{1,3,4}$.
  Thus $\abs*{\Fd(K_{1,3,4})} = 1 + 3 = 4$.

  However, the first $\TY$ on $K_{1,3,4}$ produces a non-trivial degree three vertex.
  If $(v_1, w_1, x_1)$ are the vertices of the triangle, then $x_1$ becomes a degree three vertex in graph 2.
  Making a $\YT$ move at $x_1$ produces graph 4.
  Up to symmetry, there's a unique triangle in graph 4 and the resulting graph 6 has no non-trivial degree three vertices.

  The analysis above does not change for $\F(K_{1,3,c})$ if $c \geq 4$.
  There are still four graphs in $\Fd(K_{1,3,c})$, the first $\TY$ move on $K_{1,3,c}$ results in a non-trivial degree three vertex $x_1$.
  Applying the $\TY$ at $x_1$ produces a new graph that in turn admits a single $\YT$ move.
  For this reason, $\abs*{\F(K_{1,3,c})} = 6$, as required.

  It remains to treat the case where $(a,b) = 2$.
  For the remainder of this proof only, let $G = K_{2,2,4}$.
  We will proceed as in the family of $K_{1,3,4}$ above, by describing the family and then arguing that nothing changes when we add vertices to the $C$ part.
  A triangle must include a vertex from parts $A, B$, and $C$.
  Let $A = \{v_1, v_2\}, B = \{w_1, w_2\}$, and $C = \{ x_1, x_2, x_3, x_4 \}$.
  At most two triangles can involve $v_1$ and at most two triangles can involve $v_2$.
  We will use an ordered pair to indicate this.
  For example, $G = G_{(2,1)}$ indicates an element of the family where two $\TY$ moves have been performed involving $v_1$ and one $\TY$ triangle has been performed with $v_2$.
  We use superscripts to indicate that there are several ways to construct graphs with the same subscript.
  For example, there are three, non-isomorphic, $G_{(2,2)}$ graphs.

  Without loss of generality, the one or two triangles involving $v_1$ will always be $\{(v_1, w_1, x_1)\}$ and $\{(v_1, w_1, x_1),(v_1, w_2, x_2)\}$, respectively.
  Similarly, for the triangles removed containing $v_2$, the only ways to perform one or two $\TY$ moves, up to symmetry, are summarized in Table~\ref{tab:k224}.
  \begin{table}
    \centering
    \begin{tabular}{ll}
      \toprule
      Graph & Triangles Containing \(v_2\)\\
      \cmidrule(r){1-1} \cmidrule(l){2-2}
      \(G_{(1,1)}^1\) & \((v_2,w_2,x_1)\)\\[4pt]
      \(G_{(1,1)}^2\) & \((v_2,w_2,x_2)\)\\[4pt]
      \(G_{(1,1)}^3\) & \((v_2,w_1,x_2)\)\\[4pt]
      \(G_{(2,1)}^1\) & \((v_2,w_2,x_1)\)\\[4pt]
      \(G_{(2,1)}^2\) & \((v_2,w_2,x_2)\)\\[4pt]
      \(G_{(2,2)}^1\) & \((v_2,w_2,x_1), (v_2,w_1,x_4)\)\\[4pt]
      \(G_{(2,2)}^2\) & \((v_2,w_2,x_2), (v_2,w_1,x_3)\)\\[4pt]
      \(G_{(2,2)}^3\) & \((v_2,w_2,x_2), (v_2,w_1,x_2)\)\\
      \bottomrule
    \end{tabular}
    \caption{The family of $K_{2,2,4}$}\label{tab:k224}
  \end{table}

  Note that $G_{(1,1)}^3$ is isomorphic to $G_{(2,0)}$, so that these, along with $G_{(0,0)}, G_{(1,0)},$ and $G_{(2,0)}$, give us ten graphs.
  We now argue that these ten graphs give us $\Fd(G)$, and that $\Fd(G) = \F(G)$.
  The family is depicted in Figure \ref{fig:k22c family}.

  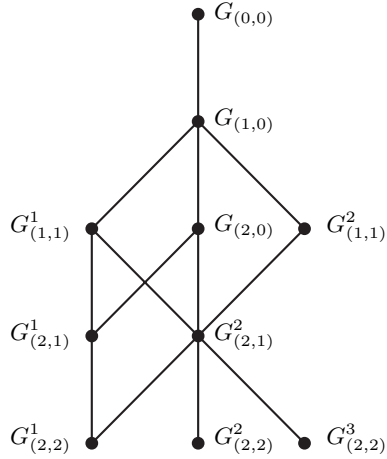
\begin{figure}[htb]
    \labellist{}
    \small\hair 2pt
    \pinlabel $G_{(1,1)}^1$ at -21 105
    \pinlabel $G_{(2,1)}^1$ at -21 54
    \pinlabel $G_{(2,2)}^1$ at -21 5
    \pinlabel $G_{(0,0)}$ at 75 205
    \pinlabel $G_{(1,0)}$ at 75 155
    \pinlabel $G_{(2,0)}$ at 75 105
    \pinlabel $G_{(2,1)}^2$ at 75 54
    \pinlabel $G_{(2,2)}^2$ at 75 5
    \pinlabel $G_{(1,1)}^2$ at 128 105
    \pinlabel $G_{(2,2)}^3$ at 128 5
    \endlabellist{}
    \centering
    \includegraphics[scale=0.8]{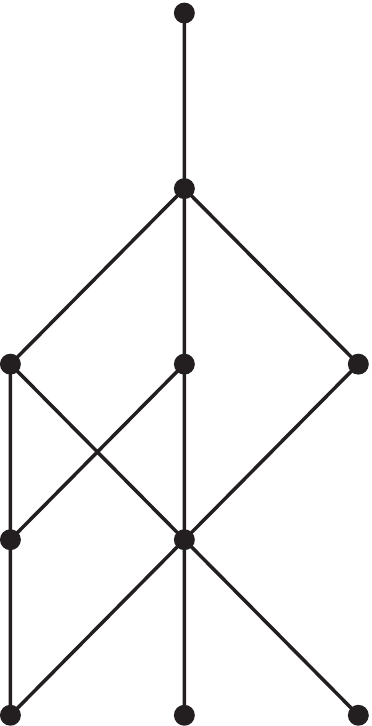}
    \caption{The $K_{2,2,c}$ family}\label{fig:k22c family}
  \end{figure}

  The graphs $G_{(0,0)}$ and $G_{(1,0)}$ are the unique graphs with eight and nine vertices. The three graphs with $10$ vertices are $G_{(1,1)}^1, G_{(1,1)}^2$, and $G_{(2,0)}$ (recalling that $G_{(1,1)}^3$ is isomorphic to $G_{(2,0)}$). Of these three graphs, $G_{(1,1)}^1$ is the unique one with a vertex of degree two and $G_{(2,0)}$ is the unique one with a vertex of degree six, so these three graphs are non-isomorphic.

  There are two graphs of degree $11$, $G_{(2,1)}^1$ and $G_{(2,1)}^2$, but only $G_{(2,1)}^1$ has a vertex of degree two.

  Finally, there are three graphs with twelve vertices, $G_{(2,2)}^1, G_{(2,2)}^2$, and $G_{(2,2)}^3$. Of these, $G_{(2,2)}^2$ is the only one with a vertex of degree two, while $G_{(2,2)}^1$ has five vertices of degree four, whereas $G_{(2,2)}^3$ has only four. This shows that $|\Fd(G)| = 10$.

  We now show that $\Fd(G) = \F(G)$.
  Note that $G_{(0,0)}, G_{(1,0)}, G_{(2,0)}, G_{(1,1)}^1$, and $G_{(2,1)}^1$ have no non-trivial degree three vertices.
  The graph $G_{(1,1)}^2$ has two non-trivial degree three vertices, $x_1$ and $x_4$.
  Performing a $\YT$ move on either yields $G_{(1,0)}$.
  Similarly, for $G_{(2,1)}^2$, performing a $\YT$ on $x_1$ or $x_4$ yields $G_{(2,0)}$ and $G_{(1,1)}^2$, respectively.
  On $G_{(2,2)}^1$, we may perform a $\YT$ move on either $x_3$ or $x_4$, which would result in $G_{(2,1)}^1$ and  $G_{(2,1)}^2$.
  For $G_{(2,2)}^2$, non-trivial degree three vertices are $x_1, x_2, x_3$, and $x_4$.
  A $\YT$ on any of them gives $G_{(2,1)}^2$.
  Finally, the non-trivial degree three vertices for $G_{(2,2)}^3$ are $x_1, x_4$ and $x_2$, and a $\YT$ on any of them yields $G_{(2,1)}^2$.
  This gives that $\Fd(G) = \F(G)$.

  Similar to the $K_{1,3,4}$ case, notice that nothing in this argument changes if we replace $G$ with $K_{2,2,c}$ for $c > 4$.
\end{proof}

\begin{lemma}\label{lem33case}
  If $c \geq 9$, then $\abs*{\F(K_{3,3,c})} = \abs*{\F(K_{3,3,9})}$
\end{lemma}

\begin{proof}
  With the aid of a computer, we verify that $\abs*{\F(K_{3,3,9})}=298$ and $\abs*{\Fd(K_{3,3,9})} = 237$.
  By Theorem~\ref{thmFstab}, for $c \geq 9$, $|\Fd(K_{3,3,c})| = |\Fd(K_{3,3,9})| = 237$.
  We must show that the remaining 61 graphs of $\F(K_{3,3,9})$ can be identified uniquely with those of $\F(K_{3,3,c})$ whenever $c \geq 9$.

  For this, we note that there are three additional graphs in $\F(K_{3,3,9})$ that are Y-free; they have no degree three vertices.
  We denote them as $G_{17}$, $G_{19}$, and $G_{21}$, where the subscript corresponds to the order (number of vertices, all graphs in the family have size 63).
	In other words, $\F(K_{3,3,9}) = \Fd(K_{3,3,9}) \cup \Fd (G_{17}) \cup \Fd (G_{19}) \cup \Fd(G_{21})$.
  Our strategy is to argue that there are analogous graphs $G^c_{17}$, $G^c_{19}$, and $G^c_{21}$ in $\F(K_{3,3,c})$ (for $c \geq 9$) and that the bijection between $\Fd(K_{3,3,9})$ and $\Fd(K_{3,3,c})$ extends to show the pairs $\Fd(G_i)$ and $\Fd(G^c_i)$, $i = 17,19,21$ are also in bijection.


  \begin{figure}[htb]
    \centering
    \includegraphics[scale=0.6]{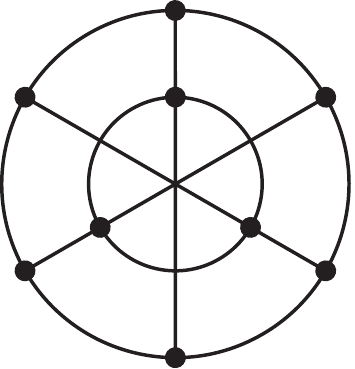}
		\hfill
    \includegraphics[scale=0.3]{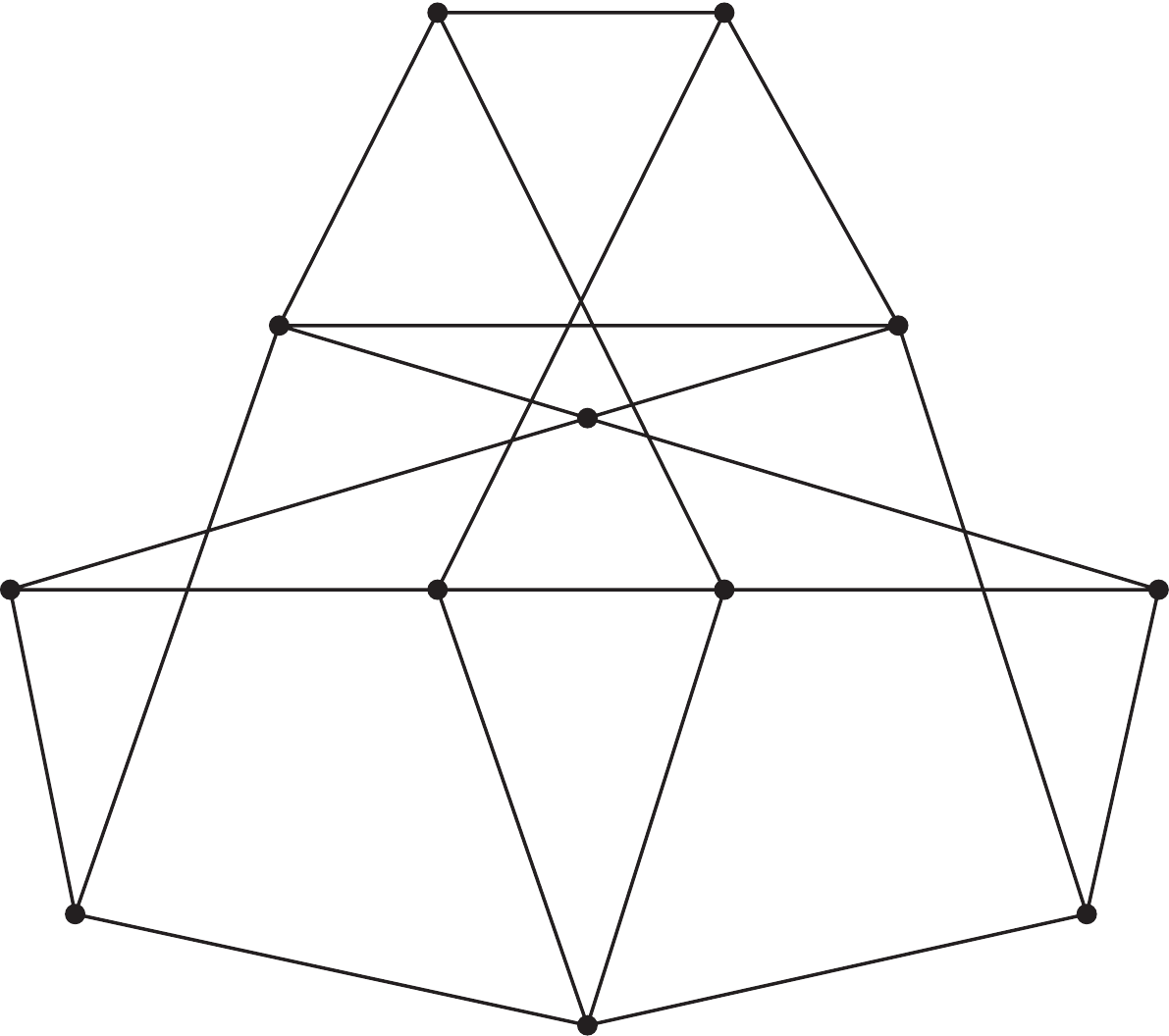}
		\hfill
    \includegraphics[scale=0.3]{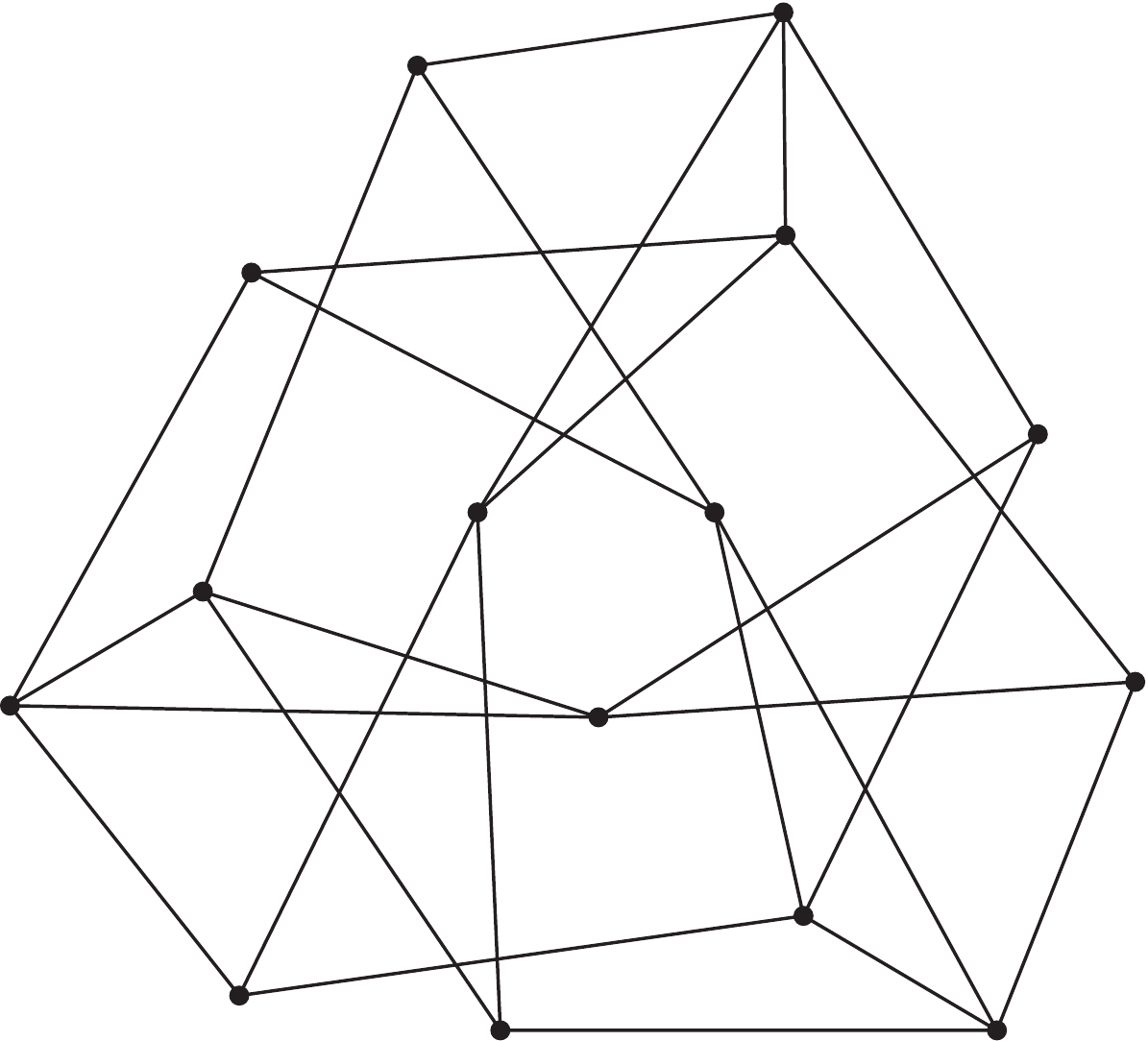}
		\caption{Subgraphs related to $K_{3,3,9}$ (L to R): $G_{17}, G_{19}$, and $G_{21}$.}\label{fig:k339}
  \end{figure}

  For this, it will be important to keep track of how the C part vertices appear in each of the Y-free graphs.
  For example, eight of the C part vertices of $K_{3,3,9}$ survive in $G_{17}$, each having degree six. The induced graph on the remaining nine vertices is $P_9$, the graph on
  nine vertices in the Petersen family $\F(K_{1,3,3})$ (see Figure~\ref{fig:k339}).
  Indeed, if we ignore eight of the C vertices of $K_{3,3,9}$, what remains is a $K_{1,3,3}$. We can identify the
  sequence of $\TY$ and $\YT$ moves as taking place in $\F(K_{1,3,3})$ while the eight C vertices maintain degree six throughout
  the sequence of moves. The neighbors of the eight $C$ vertices are the six vertices of degree three in $P_9$. 

  Then, the analogue in $\F(K_{3,3,c})$, $G^c_{17}$, consists of a $P_9$
  along with $c-1$ additional `part C' vertices of degree six, each adjacent to the six degree three vertices of the $P_9$. 
  In other words, for $c \geq 9$, there are at least eight part C vertices in $G^c_{17}$. As in the proof of
  Theorem~\ref{thmFstab}, to show that $\Fd(G_{17})$ is in bijection with $\Fd(G^c_{17})$, it is enough to observe that there are 
  at most eight edge-disjoint triangles in $G_{17}$ (or $G^c_{17}$) that make use of part C vertices.
  In fact there are only six edges between degree three vertices of $P_9$, which is less than eight.
  Therefore, the bijection of 
  Theorem~\ref{thmFstab} extends and shows $\Fd(G_{17})$ is in bijection with $\Fd(G^c_{17})$.

  For graph $G_{19}$, there are seven part C vertices, each of degree six. The induced graph, $H_{12}$ on the remaining 12 vertices has 21 edges and is shown in
  Figure~\ref{fig:k339}. The seven part C vertices are adjacent to each of the six degree three vertices in $H_{12}$. To show that $\Fd(G_{19})$ is in bijection
  with $\Fd(G^c_{19})$, it is enough to observe that there are at most seven edges in $H_{12}$ between degree three vertices. In fact, there are only three.

  Finally, for $G_{21}$, six part C vertices remain, each of degree six. The induced graph, $H_{15}$ (see Figure~\ref{fig:k339}) on the other 15 vertices has 27 edges.
  The part C vertices are adjacent to each of the six degree three vertices in $H_{15}$. There are no longer any edges directly connecting any pair of degree three vertices
  in $H_{15}$, so we again have the required bijection between the graphs of $\Fd(G_{21})$ and $\Fd(G^c_{21})$.
\end{proof}

\begin{lemma}\label{lem23case}
If $c \geq 6$, then $|\F(K_{2,3,c})| = |\F(K_{2,3,6})|$.
\end{lemma}

\begin{proof}
  The idea is the same as in Lemma~\ref{lem33case}.
  With the aid of a computer, we have that $\abs*{F(K_{2,3,6})} = 97$ and $\abs*{\Fd(K_{2,3,6})} = 30$ and so there are then $67$ graphs in $\F(K_{2,3,6}) \setminus \Fd(K_{2,3,6})$.
  There are seven graphs in $\F(K_{2,3,6}) \setminus \Fd(K_{2,3,6})$ that have no degree three vertices.
  A summary of the properties of these graphs is given in Table~\ref{k236_seven}.
  Subscripts indicate the number of vertices in the graph.
    
  \begin{table}[]
    \centering
    \begin{tabular}{ccc}
      \toprule
      Graph $G$ & Part $C$ Survivors & $\abs*{\Fd(G)}$\\
      \cmidrule(r){1-1} \cmidrule(lr){2-2} \cmidrule(l){3-3}
      \(G_{12}\) & 5 & 51\\
      \(I_{12}\) & 4 & 29\\
      \(G_{13}\) & 4 & 18\\
      \(H_{13}\) & 4 & 19\\
      \(I_{13}\) & 4 & 16\\
      \(J_{13}\) & 4 & 4\\
      \(G_{14}\) & 3 & 4\\
      \bottomrule
    \end{tabular}
    \caption{Graphs in $\F(K_{2,3,6}) \setminus \Fd(K_{2,3,6})$ Without Degree Three Vertices}\label{k236_seven}
  \end{table}

  Since $\F(K_{2,3,6}) = \Fd(K_{2,3,6}) \cup \Fd(G_{12}) \cup \Fd(I_{12}) \cup \Fd(G_{13}) \cup \Fd(H_{13}) \cup \Fd(I_{13}) \cup \Fd(J_{13}) \cup \Fd(G_{14})$, we will again argue that their are analogous graphs $X^c$ for $X \in \{G_{12}, I_{12}, G_{13}, H_{13}, I_{13}, J_{13}, G_{14}\}$ such that the bijection between $\Fd(K_{2,3,6})$ and $\Fd(K_{2,3,c})$ for $c \geq 6$ extends to a bijection between $\Fd(X)$ and $\Fd(X^c)$ for $X \in \{G_{12}, I_{12}, G_{13}, H_{13}, I_{13}, J_{13}, G_{14}\}$.

  In the case of $G_{12}$, five of the part $C$ vertices of $K_{2,3,6}$ survive in $G_{12}$, each with degree five.
  Deleting these five vertices give us the subgraph in Figure~\ref{fig:k236}.
  The part $C$ vertices are each adjacent to all of the vertices in this subgraph except the ``top-left'' and ``bottom-right'' vertices of degree four.
  Thus the analogue in $\F(K_{2,3,c})$, $G_{12}^c$, consists of Figure~\ref{fig:k236} along with $c-1$ additional part $C$ vertices, with the same adjacencies.
  Since there are four edge-disjoint triangles involving part $C$ vertices in either $G_{12}$ or $G_{12}^c$, the bijection in Theorem~\ref{thmFstab} extends to a bijection between $\Fd(G_{12})$ and $\Fd(G_{12}^c)$.


	\begin{figure}
	  \centering
		\includegraphics[scale=0.3]{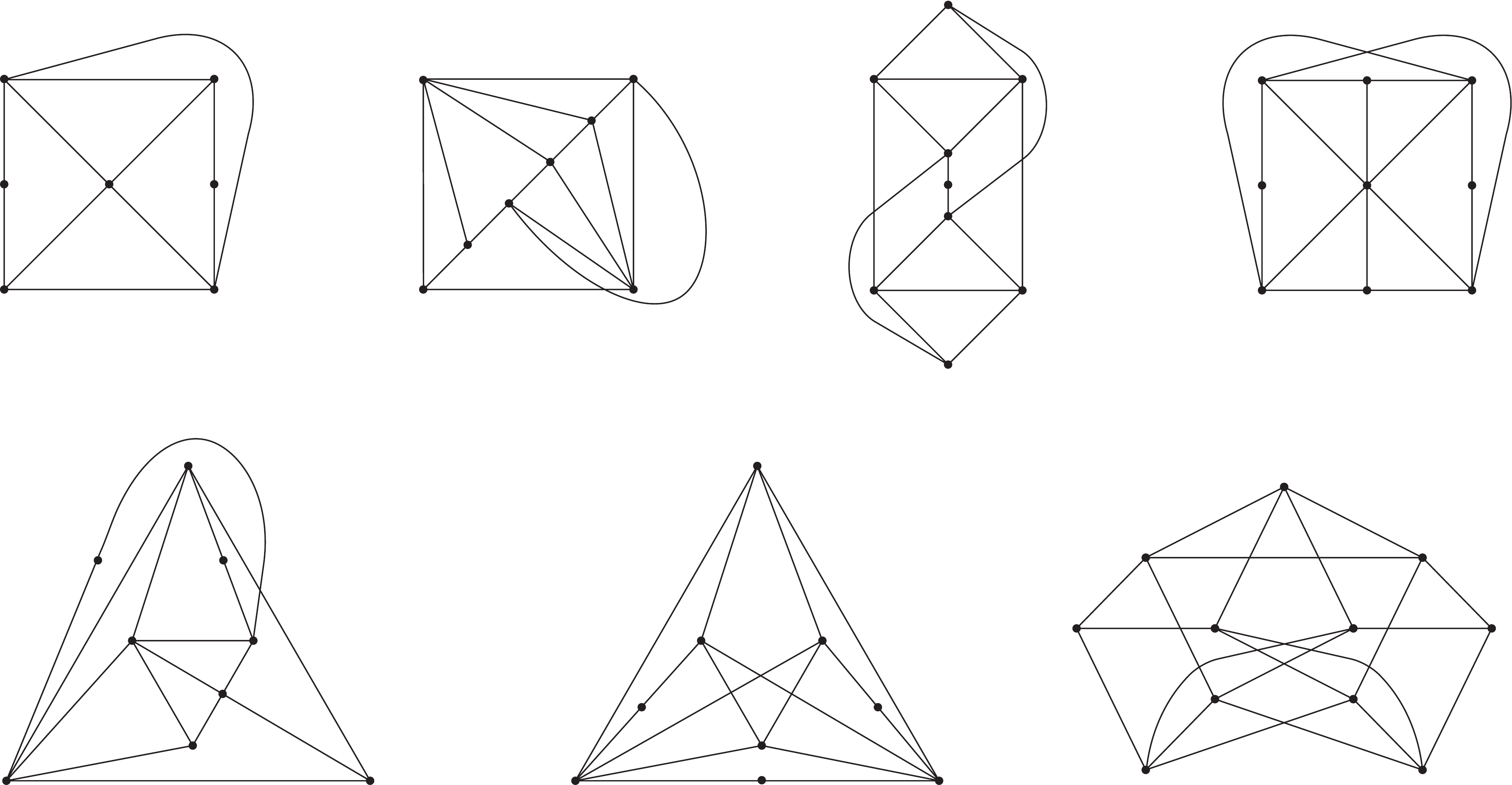}
		\caption{Subgraphs related to $K_{2,3,6}$. Top row (L to R): $G_{12}, I_{12}, G_{13}, I_{13}$; bottom row (L to R): $H_{13}, J_{13}, G_{14}$.}\label{fig:k236}
	\end{figure}

  The other six cases are similar. The subgraphs resulting from removing the part $C$ vertices are depicted in Figure~\ref{fig:k236}.
\end{proof}

\begin{proof} (of Theorem~\ref{thmtri}) 
  Combining Lemmas~\ref{lemFdeqF} and \ref{lemb3} with Theorem~\ref{thmFstab} establishes the theorem if either $a+b>6$ or $b > 3$. Lemmas~\ref{lem3cases}, \ref{lem33case}, and \ref{lem23case} handle the remaining cases.
\end{proof}

%
%

\section{Families of multipartite graphs stabilize.}

It is straight-forward to alter the arguments in the preceding section to multipartite graphs. We do so now.

\begin{lemma}\label{lem6sum}
  Let $1 \leq a_1 \leq \cdots \leq a_n$ and $a_1 + \cdots + a_{n-1} > 6$, then
  \[
      \abs*{\mathcal{F}_{\Delta}(K_{a_1, \ldots, a_n})} = \abs*{\mathcal{F}(K_{a_1, \ldots, a_n})}.
  \]
\end{lemma}

\begin{proof}
  The argument is identical to the proof of Lemma~\ref{lemFdeqF}.
  Let $G = K_{a_1, \ldots, a_n}$ and $A_1, A_2, \ldots, A_n$ be a partition of $V(G)$ with each $|A_i| = a_i$.
  A $\TY$ move will produce only trivial degree three vertices.
  The vertices of least degree are those in $A_n$, which have degree $a_1 + \cdots + a_{n-1}$.
  Since $\TY$ moves can at most halve the degree of a vertex in $A_n$ and these have degree greater than $6$, the only degree three vertices in a descendant of $K_{a_1, \ldots, a_n}$ are the trivial ones.
\end{proof}

If there are at least seven parts, then the sum of the $a_i$'s will automatically exceed six. So the next lemma follows immediately from the last.

\begin{lemma}
  Let $1 \leq a_1 \leq \cdots \leq a_n$ and $n > 6$, then
  \[
      \abs*{\mathcal{F}_{\Delta}(K_{a_1, \ldots, a_n})} = \abs*{\mathcal{F}(K_{a_1, \ldots, a_n})}.
  \]
\end{lemma}

\begin{thm}\label{thmmpstab}
  Let $1 \leq a_1 \leq \cdots \leq a_n$ and $e = \#E(K_{a_1, \ldots, a_{n-1}})$. If $a_n \geq e$, then
  \[
  \abs*{\mathcal{F}_{\Delta} (K_{a_1, \ldots, a_n})} = \abs*{\mathcal{F}_{\Delta} (K_{a_1, \ldots, a_{n-1}, e})}.
  \]
\end{thm}
\begin{proof}
  The proof is identical to Theorem~\ref{thmFstab}.
  Every element of $\Fd$ is achieved from $K_{a_1, \ldots, a_n}$ by a series of $m$ $\TY$ moves on edge-disjoint triangles.
  Let $H \in \Fd$ be given by $\TY$ moves on disjoint triangles $(\alpha_1, \beta_1, \gamma_1), \ldots, (\alpha_m, \beta_m, \gamma_m)$.
  The introduced degree three vertices $y_1, \ldots, y_m$ cannot be a part of a triangle in $H$, so there is a bijection between sequences of triangles in $K_{a_1, \ldots, a_n}$ and elements of $\Fd(K_{a_1, \ldots, a_n})$.
  Therefore the maximum length of such a sequence is given by $\#E(K_{a_1, \ldots, a_{n-1}})$.

  We now provide injective maps between $\Fd(K_{a_1, \ldots, a_{n-1}, e})$ and $\Fd(K_{a_1, \ldots, a_n})$.
  Let $H \in \Fd(K_{a_1, \ldots, a_n})$ and let $x_1, \ldots, x_m$ be the vertices from $A_n$ appearing in its associated sequence of edge-disjoint triangles.
  Extend the labeling of vertices of $A_n$ so that $A_n = \{x_1, x_2, \ldots, x_m, x_{m+1}, \ldots, x_{a_n}\}$.
  By deleting vertices $\{x_{e+1}, x_{e+2}, \ldots, x_{a_n}\}$, we identify $H$ with an element of $\Fd(K_{a_1, \ldots, a_{n}})$.
  It is clear that adding vertices to an element of $\Fd(K_{a_1, \ldots, e})$ will give an element of $\Fd(K_{a_1, \ldots, a_n})$.
\end{proof}

Combining Lemma~\ref{lem6sum} and Theorem~\ref{thmmpstab} gives our main theorem, Theorem~\ref{thmmain}.

%
%

\section{Multipartite graph families that don't stabilize.}

We have encountered four types of complete multipartite graph whose family sizes do not appear to stabilize: $K_n$, $K_{3,y}$, $K_{1,2,c}$ and $K_{1,1,1,y}$. 
Since a single $\YT$ move on $K_{3,y}$  gives $K_{1,1,1,y-1}$,
\[
    \F(K_{3,y}) = \F(K_{1,1,1,y-1}),
\] relating two of these four types and leaving three. In this section we motivate the exponential growth estimates mentioned in the introduction for these three types.

\begin{table}[]
  \centering
  \begin{tabular}{crrrrrrrrrrrr}
    \toprule
		$n$    & 1 & 2 & 3 & 4 & 5  & 6 & 7  & 8  & 9   & 10   & 11 & 12\\
    \cmidrule(lr){1-1} \cmidrule(lr){2-13}
		$\abs*{\mathcal{F}(K_n)}$ & 1 & 1 & 2 & 2 & 49 & 7 & 20 & 32 & 163 & 1,681 & 56,461 & 5,002,315 \\
    \bottomrule
  \end{tabular}
  \caption{Sizes of Complete Graph Families}\label{tblkn}
\end{table}

For $K_n$, the data we have collected is in Table~\ref{tblkn}.
As with the other types of graphs discussed in this section, there's an anomalous maximum at a small value, $n = 5$, after which the sizes show a steady increase for $n \geq 6$.
Let $\F_v(K_n)$ be the set of graphs in $\F(K_n)$ with exactly $v$ vertices.
A plot of $\abs{\F_v(K_n)}$ for $K_{11}$ and $K_{12}$ is given in Figure~\ref{fig:Kngrowth}.
For $K_n$ with $n \geq 8$, $\abs{\F_v(K_n)}$ seems to be well-approximated by a Gaussian with mean $3n-11$ and standard deviation $1.43 \leq \sigma \leq 1.84$.
This gives the estimate
\[
  f(n) = \frac{6}{5} {\left( 2 \pi \right)}^{3/2} e^{{(n-7)}^2/2}\,.
\]
\begin{figure}[htb]
	\centering
	\includegraphics[scale=0.45]{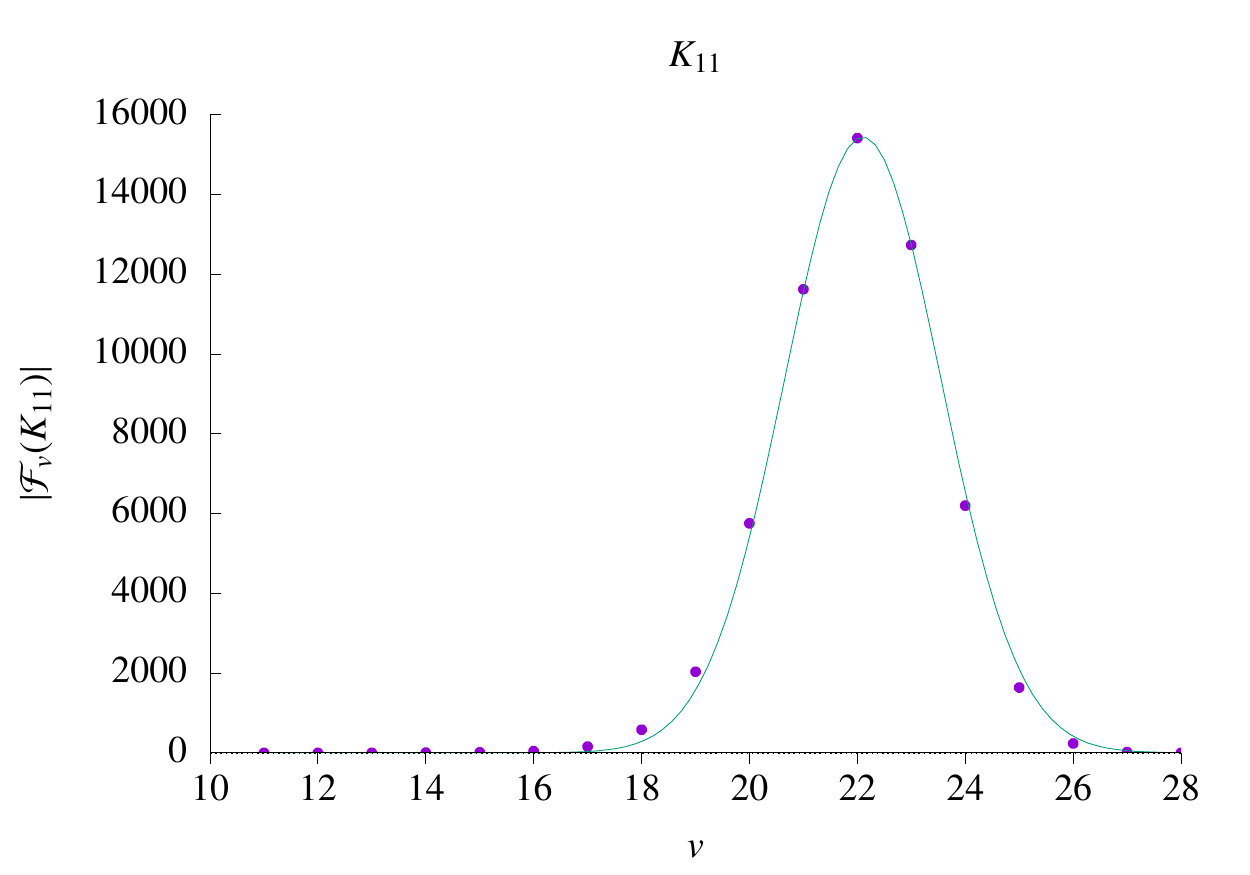}
	\includegraphics[scale=0.45]{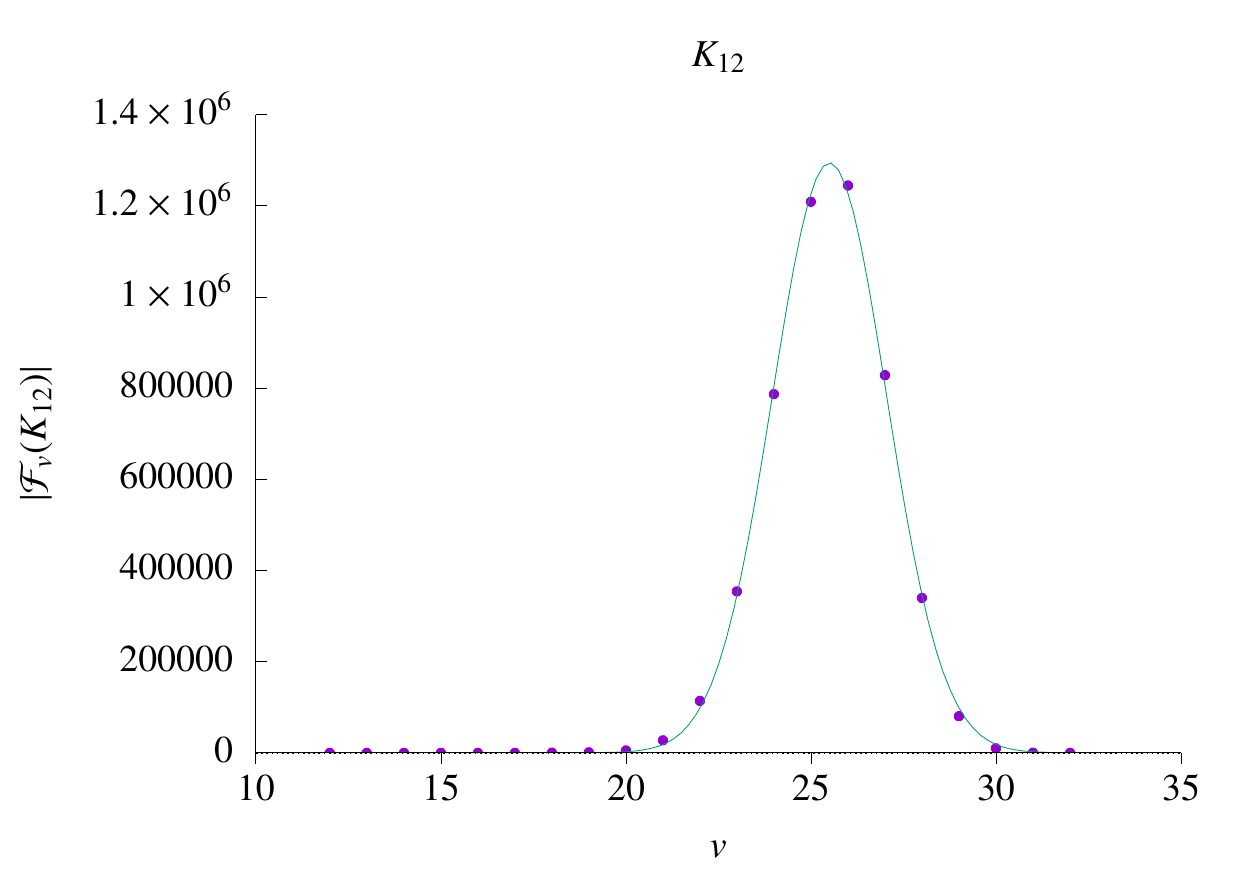}
	\caption{Number of Graphs with $v$ Vertices ($\abs{\F_v(K_n)}$) in Family of $K_n$, with Curve-Fit Gaussian}\label{fig:Kngrowth}
\end{figure}
\begin{table}[]
  \centering
  \begin{tabular}{crrrrrrrrrrr}
    \toprule
    $y$ & 1 & 2 & 3  & 4 & 5  & 6  & 7  & 8  & 9  & 10  & 11 \\
    \cmidrule(lr){1-1} \cmidrule(lr){2-12}
    $\abs*{\mathcal{F}(K_{3,y})}$ & 2 & 2 & 10 & 6 & 10 & 17 & 29 & 52 & 94 & 172 & 315 \\
    &&&&&&&&&&& \\
    $y$ & 12 & 13 & 14  & 15 & 16  &   &   &   &   &   &  \\
    \cmidrule(lr){1-1} \cmidrule(lr){2-6}
    $\abs*{\mathcal{F}(K_{3,y})}$ & 578  & 1061 & 1941 & 3533 & 6408 &  &  &  &  &  &  \\
    \bottomrule
  \end{tabular}
  \caption{Sizes of Families of Bipartite Graphs $K_{3,y}$}\label{tblk3y}
\end{table}
Table~\ref{tblk3y} shows our data for the bipartite graphs $K_{3,y}$. The values seem to follow the recursion
\[
  \abs*{\F(K_{3,y+3})} \approx \abs*{\F(K_{3,y})} + \abs*{\F(K_{3,y+1})} + \abs*{\F(K_{3,y+2})} \text{ for $y \geq 4$.}
\]
If this pattern were to persist, we would get an estimate of the form $\abs*{\F(K_{3, y+3})} = c_1 \gamma_1^y + c_2 \gamma_2^y + c_3 \gamma_3^y$
for constants $c_i$, $i = 1,2,3$, where $\gamma_i$ are the roots of $x^3 = x^2+ x + 1$. In modulus, the largest root is the real root, which is close to $e^{0.61}$.
This suggests that $\abs*{\F(K_{3,y})}$ has a bound of the form $a e^{0.61}$. Fitting the data for $y \geq 4$ to $ae^{by}$ gives $a \approx 2.68$, $b \approx 0.599$.
Rounding $b$ to $\frac35$, we approximated $a$ by $\frac83$ to get the upper bound proposed in Question 2. We've verified that the proposed inequality is
valid for $4 \leq y \leq 13$.

\begin{table}[]
  \centering
  \begin{tabular}{crrrrrrrrrr}
    \toprule
    $c$ & 1 & 2 & 3 & 4 & 5  & 6  & 7  & 8  & 9  & 10  \\
    \cmidrule(lr){1-1} \cmidrule(lr){2-11}
    $ \abs*{\F(K_{1,2,c})}$ & 2 & 3 & 21 & 14 & 22 & 40 & 78 & 153 & 299 & 581 \\
    \bottomrule
  \end{tabular}
\caption{Sizes of Tripartite Graphs $K_{1,2,c}$}\label{tblk12c}
\end{table}

Table~\ref{tblk12c} displays our calculations for the final type of graph, $K_{1,2,c}$. Similar to the previous case, for $c \geq 4$, it appears that 
$|\F(K_{1,2,c+4})|$ is approximately the sum of the previous four terms. Then, the size should grow exponentially with the largest root of
$x^4 = x^3+x^2+x+1$, which is a real root near $e^{0.656}$. Fitting the data for $c \geq 4$ to $ae^b$ gives $a \approx 5.5$ and 
$b \approx 0.67$. Rounding $b$ to $\frac23$, we approximated $a$ by $\frac{16}{3}$ to get the lower bound proposed in Question 3. 

%
%

\section{Precise bounds for simple families.}

In this section we prove three theorems that give precise calculations of size for some simple families. We also state a conjecture.

\setcounter{thm}{2}


\begin{proof}[Proof of Theorem~\ref{thmbip}]
No $\TY$ or $\YT$ moves are possible, so this is clear.
\end{proof}


\begin{proof}[Proof of Theorem~\ref{thmK1bc}]
Let $G = K_{1,b,c}$ with vertices given by $A = \{v\}$, $B = \{w_1, \ldots, w_b\},$ and $C = \{x_1, \ldots, x_c\}$. Since the minimum degree possible is $7$, by previous arguments, we need only consider sequences of edge-disjoint triangles $(v, w_1, x_1), \ldots, (v, w_n, x_n)$ whose corresponding TY moves result in non-isomorphic graphs. Note that we must have $w_i \neq w_j$ for $i \neq j$ since each triangle must go through $v$. Thus we have $b$ sequences which result in distinct graphs. Adding in $K_{1,b,c}$ itself gives the desired result.
\end{proof}

The following lower bound for the $K_{2,b,c}$ family is surprising in that the growth of $g(b,c)$ is quite close to the observed growth of $\abs*{\F(K_{2,b,c})}$ (discussed in more detail below).
Recall that $P(x,y,z)$ is the set of partitions of $z$ into two parts bounded by $x$ and $y$ and
\[
  g(b,c) = 5 + \sum_{i=2}^b \sum_{j=0}^i \left( \abs*{P(i, b-i, j)} \cdot \abs*{P(i, c-i, j)} \right).
\]
 


\begin{proof}[Proof of Theorem~\ref{thmK2bc}]
  The proof of Theorem~\ref{thmFstab} gives us a way to determine a bound on the size of $\F(K_{2,b,c})$. We need a lower bound on the number of sequences of edge-disjoint triangles in $K_{2,b,c}$ such that corresponding $\TY$ moves on these sequences of disjoint triangles result in non-isomorphic graphs.

  A triangle must have a vertex in parts $A, B$, and $C$. Let $A = \{v_1, v_2\}$. At most $b$ triangles can contain $v_1$, and at most $b$ triangles can contain $v_2$.

  The proof proceeds as follows: we first describe a method of choosing a sequence of triangles on which we will perform $\TY$ moves.
	We identify each triangle in the sequence with its vertices.
	Then, we argue that no two distinct such choices give isomorphic graphs.

  Suppose our sequence of $n$ triangles is such that $n = i+j$ with $2 \leq i \leq b$, $0 \leq j \leq i$ where $i$ is the number of triangles involving $v_1$.
	Fix a labeling of the $B$ and $C$ vertices such that the $i$ triangles with a $v_1$ vertex are $(v_1, w_\alpha, x_\alpha)$ for $1 \leq \alpha \leq i$.
	Partition $B$ and $C$ based on these choices. Define
  \begin{align*}
    B_1 &= \{ w_1, \ldots, w_i \}, \quad B_2 = \{w_{i+1}, \ldots, w_b\} \\
    C_1 &= \{ x_1, \ldots, x_i \}, \quad C_2 = \{x_{i+1}, \ldots, x_c\}.
  \end{align*}
  Thus $B_1$ and $C_1$ are the vertices in triangles including the $v_1$ vertex.

  Since $n = i + j$, then $j \leq i$ is the number of triangles in our sequence that include vertex $v_2$.
	For each of the $j$ triangles involving $v_2$, we must pick an element of either $B_1$ or $B_2$ and an element of either $C_1$ or $C_2$.
	The number of ways to choose $j$ triangles in this way is given by $\abs*{P(j,b-j,i)} \cdot \abs*{P(j,c-j,i)}$.
	We'll assume $i \geq 2$, so that there remain enough edges between $B_1$ and $C_1$ to form the $j$ triangles on $v_2$.
	Indeed, there are $i^2$ edges between $B_1$ and $C_1$ and $i$ of them are used for the triangles containing $v_1$.
	Assuming $i \geq 2$, there remain $i^2 - i \geq i \geq j$ edges.
	It's easy to check that there's one way to form a graph when $i = 0$ and four for $i = 1$.

  Define $G = G(s, t)$ to be the graph obtained by performing $\TY$ moves on the $n=i+j$ triangles, parameterized by $s$ and $t$ as follows: our sequence of edge-disjoint triangles includes the $i$ triangles $(v_1, w_\alpha, x_\alpha)$ for $1 \leq \alpha \leq i$ and the $j$ triangles $\{(v_2, w_{\beta_1}, x_{\gamma_1}), \ldots, (v_2, w_{\beta_j}, x_{\gamma_j})\}$, where $w_{\beta_1}, \ldots, w_{\beta_{s}} \in B_1$, $w_{\beta_{s+1}}, \ldots, w_{\beta_{j}} \in B_2$ and similarly $x_{\gamma_1}, \ldots, x_{\gamma_{t}} \in C_1$, $x_{\gamma_{t+1}}, \ldots, x_{\gamma_{j}} \in C_2$ with $0 \leq s, t \leq j$. If $s = 0$, then all $w_{\beta}$ vertices are in $B_2$ and similarly for $t = 0$.

  To identify whether or not two such graphs might be isomorphic, let's identify the degrees of the vertices in $G(s,t)$.
	In $K_{2,b,c}$ there are two vertices of degree $b+c$, $b$ of degree $c+2$, and $c$ of $b+2$.
	After the $i$ $\TY$ moves on triangles with a $v_1$ vertex, $v_2$ still has degree $b+c$, $v_1$ will have degree $b+c-i$, there are $i$ in $B_1$ of degree $c+1$ and the remaining $b-i$ vertices in $B_2$ of degree $c+2$.
	Similarly, the $i$ vertices of $C_1$ have degree $b+1$ and the $c-i$ vertices in $C_2$ remain at $b+2$.
	Finally, we have added $i$ degree $3$ vertices.
	After a further $j$ $\TY$ moves, $G(s,t)$ has the following degrees and counts: one of degree $b+c-i$, one of degree $b+c-j$, $s$ of $c$, $i+j-2s$ of $c+1$, $b+s-i-j$ of $c+2$, $t$ of $b$, $i+j-2t$ of $b+1$, $c+t-i-j$ of $b+2$, and $i+j$ vertices of degree $3$.

  We will argue that two such  graphs $G_1 = G(s_1, t_1)$ and $G_2 = G(s_2,t_2)$ can be isomorphic only if $(i_1, j_1, s_1, t_1) = (i_2, j_2, s_2, t_2)$; the four constants must agree.
	We note that the theorem holds if $b = 3$ as illustrated by Tables~\ref{fig:k2xy} and \ref{fig:g2xy} below.
	Both $g(3,c)$ and $|\F(K_{2,3,c})|$ stabilize for $c \geq 6$, so it is enough to verify the result for $4 \leq c \leq 6$.
	So, we will assume  $b > 3$.
	Counting the vertices of degree $3$ we have $i_1+j_1 = i_2 + j_2$ and the vertices of degree $b$ show that $t_1 = t_2$.
	We can identify $v_1$ and $v_2$ as the two vertices that, between them, are adjacent to all the degree three vertices.
	Comparing the degrees of $v_1$ and $v_2$, since $j \leq i$ (if $i = j$, then they are interchangeable), we can identify the $i$'s and $j$'s, which shows $i_1 = i_2$ and $j_1 = j_2$.

  It remains to argue $s_1 = s_2$.
	Ordinarily, this can be done by comparing the vertices of degree $c$.
	However, there may be additional vertices of degree $c$ beyond the $s$ that we expect.
	For example, if $c = b+1$, we would have $s + i+j - 2t$ vertices of degree $c$.
	Since we've already shown the other three constants agree, comparing the vertices of degree $c$ still will give us the required $s_1 = s_2$.
	Similarly if $c = b+2$, the additional $c+t-i-j$ vertices of degree $c$ cause no problem as we've already established that this number is the same for both graphs.
	It may be that $v_1$ or $v_2$ have degree $c$, but we've discussed how to identify these vertices and, for the graphs to be isomorphic, their degrees must agree in $G_1$ and $G_2$.
\end{proof}

Data for both $|\F(K_{2,x,y})|$ and $g(x,y)$ is given in Tables~\ref{fig:k2xy} and \ref{fig:g2xy}, respectively.
Based on the table values, it appears that $g(x,2x) = g(x,2x-1) + 1$, which corresponds to the pattern $|\F(K_{2,x,2x})| = |\F(K_{2,x,2x-1})| + 1$ that we observe for $3 \leq x \leq 6$ (and conjecture for greater $x$, see below).
The growth patterns of the two functions are similar in many respects.
For example, we have shown in Theorem~\ref{thmFstab} that the size of the graph family of $K_{2,x,y}$ stabilizes at $K_{2,x,2x}$ and, in Table~\ref{fig:g2xy}, $g(x,y)$ shows a similar stabilization.

\begin{table}[htb]
  \centering
  \begin{tabular}{rcrrrrrrrrr}
    \toprule
    $x$ & & \multicolumn{9}{c}{$y$}\\
    \cmidrule{3-11}
     & & 4 & 5 & 6 & 7 & 8 & 9 & 10 & 11 & 12\\
    \midrule
    3 & & 93 & 96 & 97 & 97 & 97 & 97 & 97 & 97 & 97\\
    4 & & 43 & 70 & 78 & 80 & 81 & 81 & 81 & 81 & 81\\
    5 & & 70 & 96 & 166 & 184 & 192 & 194 & 195 & 195 & 195\\
    6 & & 78 & 166 & 215 & 380 & 428 & 447 & 455 & 457 & 458\\
    7 & & 80 & 184 & 380 & 450 & 827 & 931 & 981 & 1000 & 1008\\
    \bottomrule
  \end{tabular}
  \caption{$|\F(K_{2,x,y})|$\label{fig:k2xy}}
\end{table}

\begin{table}[htb]
  \centering
  \begin{tabular}{rcrrrrrrrrr}
    \toprule
    $x$ & & \multicolumn{9}{c}{$y$}\\
    \cmidrule{3-11}
     & & 4 & 5 & 6 & 7 & 8 & 9 & 10 & 11 & 12\\
    \midrule
    3 & & 23 & 25 & 26 & 26 & 26 & 26 & 26 & 26 & 26\\
    4 & & 37 & 45 & 50 & 52 & 53 & 53 & 53 & 53 & 53\\
    5 & & 45 & 65 & 79 & 87 & 92 & 94 & 95 & 95 & 95\\
    6 & & 50 & 79 & 109 & 129 & 143 & 151 & 156 & 158 & 159\\
    7 & & 52 & 87 & 129 & 169 & 199 & 219 & 233 & 241 & 246\\
    \bottomrule
  \end{tabular}
  \caption{$g(x,y)$\label{fig:g2xy}}
\end{table}

We conclude this section with a conjecture.

\begin{conj}
  Let $n \geq 3$, $1 \leq a_1 \leq \cdots \leq a_n$, and $e = \#E(K_{a_1, \ldots, a_{n-1}})$.
  If $a_1 + \cdots + a_{n-1} > 6$ and $a_n \geq e$, then
  \[
      \abs*{\mathcal{F} (K_{a_1, \ldots, a_{n-1}, e-1})} = \abs*{\mathcal{F} (K_{a_1, \ldots, a_{n-1}, a_n})} - 1.
  \]
\end{conj}

The conjecture is supported by experimental data for some tripartite graphs. 
Note that for $K_{a,b,c}$, with $a \leq b \leq c$, we have $e = ab$.
Using Theorem~\ref{thmK1bc}, it is straight-forward to verify the conjecture for triples $1,b,c$.

\begin{thm}
  If $6 \leq b \leq c$, then $ \abs*{\mathcal{F} (K_{1,b, b-1})} = \abs*{\mathcal{F} (K_{1,b,c})} - 1$.
\end{thm}

\begin{proof}
  By Theorem~\ref{thmK1bc}, $\abs*{\F (1,b,c)} = 1+b$.
  We must show that $\abs*{\F(1,b,b-1)} = b$.
  If $b > 6$, the same theorem shows $\abs*{\F(1,b,b-1)} = \abs*{\F(1,b-1,b)} = b$, as required.
  All that remains is the easy verification that, when $b = 6$,
  $\abs*{\F(1,6,5)} = \abs*{\F(1,5,6)} = 6$.
\end{proof}

In addition to the triples covered by the theorem above, Table~\ref{fig:k2xy} shows the conjecture holds for $(a,b) = (2,5)$ or $(2,6)$.
Using a computer, we have also verified the case $(a,b) = (3,4)$.

\bibliographystyle{plain}
\bibliography{MultiGraphs}

\end{document}